\documentclass{amsart}
\usepackage{amssymb}
\usepackage{graphicx}
\usepackage{enumerate}
\usepackage{multirow}
\usepackage{amsmath,color}
\usepackage{hyperref}
\usepackage{url}

\newtheorem{theorem}{Theorem}[section]

\newtheorem{proposition}[theorem]{Proposition}

\newtheorem{lemma}[theorem]{Lemma}
\newtheorem{corollary}[theorem]{Corollary}

\def\R{\mbox{$\mathbb{R}$}}
\def\F{\mbox{$\mathbb{F}$}}

\newcommand{\ep}{\epsilon}
\newcommand{\la}{\lambda}
\newcommand{\G}{\Gamma}

\newcommand{\AL}{\mathbb{A}}

\def\rk{\mathop{\rm rk }\nolimits}
\def\dist{\mathop{\rm dist }\nolimits}

\begin{document}

\title{Geometric aspects of $2$-walk-regular graphs}

\author{Marc C\'amara}
\address{Department of Econometrics and O.R., Tilburg University,
	P.O. Box 90153, 5000 LE Tilburg, The Netherlands}
\email{marc.camara@gmail.com}
\author{Edwin R. van Dam}
\address{Department of Econometrics and O.R., Tilburg University,
	P.O. Box 90153, 5000 LE Tilburg, The Netherlands}
\email{Edwin.vanDam@uvt.nl}
\author{Jack H. Koolen}
\address{School of Mathematical Sciences,
University of Science and Technology of China, 96 Jinzhai Road, Hefei, 230026, Anhui, P.R. China}
\email{koolen@postech.ac.kr}
\author{Jongyook Park}
\address{School of Mathematical Sciences,
University of Science and Technology of China,
96 Jinzhai Road, Hefei, 230026, Anhui, P.R. China}
\email{jongyook@hanmail.net}
\thanks{The work of Marc C\'amara was (partly) financed by the Netherlands Organisation for Scientific Research (NWO) and by the Ministerio de Educaci\'{o}n y Ciencia of Spain under project MTM2011-28800-C02-01.
This work was done while Jongyook Park was a postdoctoral fellow at Tilburg University, (partly) financed by the
mathematics cluster DIAMANT of the Netherlands Organisation for Scientific Research (NWO), and while Jack Koolen was
visiting Tilburg University and Tohoku University, supported by an NWO visitor grant and a JSPS grant, respectively.
Jack Koolen was also partially supported by a grant of the `One hundred talent program' of the Chinese government \medskip\\
This version is published in Linear Algebra and its Applications 439 (2013), 2692-2710.}



\begin{abstract}A $t$-walk-regular graph is a graph for which the number of walks of given length between two vertices
depends only on the distance between these two vertices, as long as this distance is at most $t$. Such graphs
generalize distance-regular graphs and $t$-arc-transitive graphs. In this paper, we will focus on $1$- and in
particular $2$-walk-regular graphs, and study analogues of certain results that are important for distance-regular
graphs. We will generalize Delsarte's clique bound to $1$-walk-regular graphs, Godsil's multiplicity bound and
Terwilliger's analysis of the local structure to $2$-walk-regular graphs. We will show that $2$-walk-regular graphs
have a much richer combinatorial structure than 1-walk-regular graphs, for example by proving that there are finitely
many non-geometric $2$-walk-regular graphs with given smallest eigenvalue and given diameter (a geometric graph is the
point graph of a special partial linear space); a result that is analogous to a result on distance-regular graphs.
Such a result does not hold for $1$-walk-regular graphs, as our construction methods will show.
\end{abstract}

\maketitle

\section{Introduction}
Walk-regular graphs were introduced by Godsil and McKay \cite{Godsil198051} in their study of cospectral graphs. They
showed that the property that the vertex-deleted subgraphs of a graph $\G$ are all cospectral is equivalent to the
property that the number of closed walks of a given length $\ell$ in $\G$ is independent of the starting vertex, for
every $\ell$. They also observed that walk-regular graphs generalize both vertex-transitive graphs and distance-regular
graphs. Distance-regular graphs \cite{bcn89,DKT13} play a crucial role in the area of algebraic combinatorics, and it
was shown by Rowlinson \cite{R} that such graphs can be characterized in terms of the numbers of walks between two
vertices; in particular that this number only depends on their length and the distance between the two vertices.
Motivated by this characterization, Dalf\'o, Fiol, and Garriga \cite{fg07,DFG09} introduced $t$-walk-regular graphs;
such graphs have the property of Rowlinson's characterization at least for those vertices that are at distance at most
$t$. These $t$-walk-regular graphs were further studied by Dalf\'o, Fiol, and coauthors
\cite{DFGtcliques,DFG10,DDF11,DvDFGG11}. Dalf\'o, Van Dam, and Fiol \cite{DDF11} characterized $t$-walk-regular graphs
in terms of the cospectrality of certain perturbations, thus going back to the roots of walk-regular graphs. Dalf\'o,
Van Dam, Fiol, Garriga, and Gorissen \cite{DvDFGG11} among others raised the question of when $t$-walk-regularity
implies distance-regularity.

Our motivation for studying $t$-walk-regular graphs lies in the generalization of distance-regular graphs. In order to
better understand the latter, we would like to know which results for these graphs can be generalized to
$t$-walk-regular graphs. In this way, we aim to have a better understanding of which properties of distance-regular
graphs are most relevant.

Here we will focus on $1$- and in particular $2$-walk-regular graphs. We will for example generalize Delsarte's clique
bound \cite{D73} to $1$-walk-regular graphs. It seems however that $1$-walk-regularity is still far away from
distance-regularity, but going to $2$-walk-regularity is an important step (or jump) forward. Indeed, we will see that
several important results on distance-regular graphs have interesting generalizations to $2$-walk-regular graphs (but
not to $1$-walk-regular graphs), such as Godsil's multiplicity bound \cite{g88} and Terwilliger's analysis of the local
structure \cite{T86}. On the other hand, there are very basic construction methods for $1$-walk-regular graphs that
cannot be generalized to $2$-walk-regular graphs; indeed, most known examples of the latter come from groups as graphs
that are obtained in an elementary way (such as the line graph and halved graph) from $s$-arc-transitive graphs. We
will indeed show that $2$-walk-regular graphs have a much richer combinatorial structure than 1-walk-regular graphs. We
will show that there are finitely many non-geometric $2$-walk-regular graphs with given smallest eigenvalue and given
diameter (a geometric graph is the point graph of a special partial linear space); a result that is analogous to a
result on distance-regular graphs. In fact, this result shows that the class of $2$-walk-regular graphs is quite
limited. Again, such a result does not hold for $1$-walk-regular graphs, as our construction methods (Proposition
\ref{prop: kron}, in particular) will show.

This paper is organized as follows: in the next section, we give some technical background. In Section \ref{sec:
constructions}, we give elementary construction methods for $t$-walk-regular graphs that we will use in the remaining
sections. In Section \ref{sec:godsil}, Godsil's multiplicity bound for distance-regular graphs is generalized to
$2$-walk-regular graphs. Similarly we generalize in Section \ref{sec:terwilliger} Terwilliger's analysis of local
graphs. In Section \ref{sec:smallmult}, we study $t$-walk-regular graphs with an eigenvalue with small multiplicity.
Finally, in Section \ref{sec: geometric}, we generalize Delsarte's clique bound and study geometric $2$-walk-regular
graphs.

\section{Preliminaries}
Let $\Gamma$ be a connected graph with vertex set $V=V(\Gamma)$ and denote $x\sim y$ if the vertices $x,y\in V$ are
adjacent. The {\em distance} $\mbox{dist}_{\Gamma}(x,y)$ between two vertices $x,y\in V$ is the length of a shortest
path connecting $x$ and $y$ (we omit the index $\G$ when this is clear from the context). The maximum distance between
two vertices in $\Gamma$ is the {\em diameter} $D=D(\Gamma)$. We use $\Gamma_i(x)$ for the set of vertices at distance
$i$ from $x$ and write, for the sake of simplicity, $\Gamma(x):=\Gamma_1(x)$. The {\em degree} of $x$ is the number
$|\Gamma(x)|$ of vertices adjacent to it. A graph is {\em regular} with {\em valency} $k$ if the degree of each of its
vertices is $k$.

A graph $\Gamma$ is called {\em bipartite} if it has no odd cycle. For a connected graph $\Gamma$,
the {\em bipartite double} $\widetilde{\Gamma}$ of $\Gamma$ is the graph whose vertices are the symbols $x^{+},x^{-} (x\in V)$ and where $x^+$ is adjacent to $y^-$ if and only of $x$ is adjacent to $y$ in $\G$.

Let $\overline{\G}$ be a graph with vertex set $V(\overline{\G})$. Let $\Gamma$ be a graph whose vertices are
partitioned in $|V(\overline{\G})|$ classes of the same size. We say that $\G$ is a {\em cover} of $\overline{\G}$ if
the following three properties hold: The vertices of each class induce an empty graph in $\Gamma$; the classes give an
equitable partition in $\Gamma$ (that is, for every two classes, every vertex in one of these classes has the same
number of neighbors in the other class); and the quotient graph provided by the classes (that is, the graph on the
classes, where two classes are adjacent if there are edges (of $\G$) between them) is isomorphic to $\overline{\G}$.
This quotient graph is also called the folded graph of $\G$.

Given a graph $\G$ and $x\in V$, the {\em local graph $\Delta(x)$ at vertex $x$} is the subgraph of $\G$ induced on the
vertices that are adjacent to $x$. When all the local graphs are isomorphic, we simply write $\Delta$, and say that
$\G$ is locally $\Delta$.

For a connected graph $\Gamma$ with diameter $D$, the {\em distance-$i$ graph} $\Gamma_i$ of $\Gamma$ $(1\leq i\leq D)$
is the graph whose vertices are those of $\Gamma$ and whose edges are the pairs of vertices at mutual distance $i$ in
$\Gamma$. In particular, $\Gamma_1=\Gamma$. The {\em distance-$i$ matrix} $A_i=A_i(\Gamma)$ is the matrix whose rows
and columns are indexed by the vertices of $\Gamma$ and the ($x, y$)-entry is $1$ whenever $\mbox{dist}(x,y)=i$ and
$0$ otherwise. The {\em adjacency matrix} $A$ of $\Gamma$ equals $A_1$.

The {\em eigenvalues} of the graph $\Gamma$ are the eigenvalues of $A$. We use $\{\theta_0>\cdots>\theta_d\}$ for the
set of distinct eigenvalues of $\Gamma$. The {\em multiplicity} of an eigenvalue $\theta$ is denoted by $m(\theta)$.
Note that if $\Gamma$ is connected and regular with valency $k$, then $\theta_0=k$ and $m(\theta_0)=1$. Let
$\{v_1,\ldots,v_{m(\theta)}\}$ be an orthonormal basis of eigenvectors with eigenvalue $\theta$, and let $U$ be a
matrix whose columns are these vectors. Then the matrix $E_{\theta}=UU^{\top}$ is called a {\em minimal idempotent}
associated to $\theta$. We abbreviate $E_{\theta_i}$ by $E_i$ ($i=0,\dots,d$).

Fiol and Garriga \cite{fg07} introduced $t$-walk-regular graphs as a generalization of both distance-regular and
walk-regular graphs. A graph is {\em $t$-walk-regular} if the number of walks of every given length $\ell$ between two
vertices $x,y\in V$ only depends on the distance between them, provided that $\mbox{dist}(x,y)\leq t$ (where it is
implicitly assumed that the diameter of the graph is at least $t$). The `Spectral Decomposition Theorem' leads
immediately to
$$
A^{\ell}=\sum_{i=0}^d \theta_i^{\ell}E_i.
$$
From that, we obtain that a graph is $t$-walk-regular if and only if for every minimal idempotent the $(x,y)$-entry
only depends on $\mbox{dist}(x,y)$, provided that the latter is at most $t$ (see Dalf\'o, Fiol, and Garriga
\cite{DFG09}). In other words, for a fixed eigenvalue $\theta$ with minimal idempotent $E$, there exist constants
$\alpha_j:=\alpha_j(\theta)$ ($0\leq j\leq t$), such that $A_j\circ E=\alpha_j A_j$, where $\circ$ is the entrywise
product.

If a graph comes from one of the relations in an association scheme (see Brouwer, Cohen, and Neumaier \cite{bcn89}),
then the minimal idempotents of the graph as described above do not have to be the same as the minimal idempotents in
the scheme. Mainly for this purpose, we will develop our theory somewhat more general than seems necessary at first. In
particular, we define a {\em $t$-walk-regular idempotent} as a nonzero idempotent $E$ such that $A_j\circ E=\alpha_j
A_j$ for certain constants $\alpha_j$ ($0\leq j\leq t$). We call $E$ a {\em $t$-walk-regular idempotent for eigenvalue
$\theta$} if moreover $AE=\theta E$ holds.

If $E$ is a $t$-walk-regular idempotent with rank $m$, then clearly its diagonal elements $\alpha_0$ are positive, and
we can write $E=UU^{\top}$, where the $m$ columns of $U$ form an orthonormal basis of the eigenspace of $E$ for its
eigenvalue $1$ (note that we do not explicitly require $U$ to contain eigenvectors of $A$, as we did
above). For every vertex $x\in V$ we now denote by $\hat{x}$ the $x$-th row of $U$. The map $x\mapsto \hat{x}$ is
called a {\em representation} of $\Gamma$. Note that the vectors $\hat{x}$ ($x\in V)$ all have the same length (the
square of which is $E_{xx}=\alpha_0$); in this case we call the representation {\em spherical}. Given two vertices
$x,y\in V$, we will often refer to $u_{xy}:=E_{xy}/\alpha_0$ as the {\em $xy$-cosine}, as it can be interpreted as the
cosine of the angle between the vectors $\hat{x}$ and $\hat{y}$. We remark that if $\Gamma$ is $t$-walk-regular and
$\dist(x,y)=s\leq t$, then $u_{xy}=\alpha_s/\alpha_0$ does not depend on $x$ and $y$, but only on $s$. In this case, we
write $u_s:=\alpha_s/\alpha_0$.

Given a vertex $x$ in a graph $\Gamma$ and a vertex $y$ at distance $j$ from $x$, we consider the numbers
$a_j(x,y)=|\Gamma(y)\cap \Gamma_j(x)|$, $b_j(x,y)=|\Gamma(y)\cap \Gamma_{j+1}(x)|$, and
$c_j(x,y)=|\Gamma(y)\cap\Gamma_{j-1}(x)|$. A graph $\Gamma$ with diameter $D$ is {\em distance-regular} if these
parameters do not depend on $x$ and $y$, but only on $j$, for $0\leq j\leq D$. If this is the case then these numbers
are denoted simply by $a_j$, $b_j$, and $c_j$, for $0\leq j\leq D$, and they are called the {\em intersection numbers}
of $\Gamma$. Also, if a graph $\Gamma$ is $t$-walk-regular, then the intersection numbers are well-defined for $0\leq
j\leq t$, as they do not depend on $x$ nor on the chosen $y\in \Gamma_j(x)$ (see Dalf\'o et al.
\cite[Prop.~3.15]{DvDFGG11}). More generally, let $x$ and $y$ be two vertices at distance $h$ in a $t$-walk-regular
graph. Then the numbers $p^h_{ij}=|\Gamma_i(x)\cap\Gamma_j(y)|$ exist (i.e., they only depend on $h$, $i$ and $j$) for
nonnegative integers $h,i,j \leq t$. This follows from working out the product $A_iA_j \circ A_h$, for example; see
also Dalf\'o, Fiol, and Garriga \cite[Prop.~1]{DFG10}. Moreover, if $k_h=|\Gamma_h(x)|$, then relations such as $k_h
p^h_{ij} = k_i p^i_{hj}$ hold. From the above it is clear that a $D$-walk-regular graph is distance-regular.
For such a graph (and a minimal idempotent for $\theta$), the sequence $(u_0,\ldots,u_D)$ is known
as the {\em standard sequence} of $\Gamma$ with respect to $\theta$.

Now let $E=UU^{\top}$ be a $t$-walk-regular idempotent for eigenvalue $\theta$, with $t \geq 1$. From
$AE=\theta E$ and $U^{\top}U=I$, it follows that $AU=\theta U$, so in this case the columns of $U$ are clearly also
eigenvectors of $A$ (even though we did not require this in the definition; note also that $U$ does not necessarily
contain a basis of the eigenspace of $A$ for eigenvalue $\theta$). For the corresponding representation, this implies
that
\begin{equation}\label{eq: representation lambda}
\theta \hat{x}=\sum_{y\sim x} \hat{y}.
\end{equation}
By looking at an $(x,y)$-entry with $\mbox{dist}(x,y)=j$ in the equation $A E=\theta E$ we obtain the following
relations:
\begin{align}
k\alpha_1&=\theta \alpha_0\label{eq: relation alpha0 alpha1}\\
c_j\alpha_{j-1}+a_j\alpha_j+b_j\alpha_{j+1}&=\theta\alpha_j \qquad (1\leq j\leq t-1).\label{eq: relations alpha's}
\end{align}
In particular, $\theta$ follows from the first cosine (given $k$; and the other way around).

\section{Construction methods}\label{sec: constructions}
Highly symmetric examples of $t$-walk-regular graphs exist for $t \leq 7$ in the form of $t$-arc-transitive graphs. For
example, the infinite family of $3$-arc-transitive graphs constructed by Devillers, Giudici, Li, and Praeger
\cite{DGHP12} is also an infinite family of $3$-walk-regular graphs. Indeed, every $t$-arc-transitive graph with
diameter at least $t$ is $t$-walk-regular. By a covering construction due to Conway (see \cite[Ch.~19]{Biggs74}) and
independently Djokovi\'{c} \cite{Djokovic}, infinite families of $5$-arc-transitive graphs with valency $3$ and
$7$-arc-transitive graphs with valency $4$ were constructed. Conder and Walker \cite{CW98} also constructed infinitely
many $7$-arc-transitive graphs with valency $4$. In turn, these give rise to infinite families of cubic
$5$-walk-regular graphs and $7$-walk-regular graphs with valency $4$. The validity of the Bannai-Ito conjecture
\cite{BDKM09} (in particular the fact that there are finitely many distance-regular graphs with valency four
\cite{BK99}) for example implies that there are infinitely many $7$-walk-regular graphs that are not distance-regular.

It is worth mentioning that less-known (and less restrictive) concepts such as $t$-geodesic-transitivity and
$t$-distance-transitivity have been introduced by Devillers, Jin, Li, and Praeger \cite{DJLP11}, and both concepts are
stronger than $t$-walk-regularity.

It is rather straightforward to show that the bipartite double of a $t$-arc-transitive graph is again
$t$-arc-transitive. This could for example be applied to the infinite family of non-bipartite $2$-arc-transitive graphs
constructed by Nochefranca \cite{N91}, to obtain also an infinite family of bipartite such graphs. For $t$-walk-regular
graphs, a similar result holds, but we have to take into account the odd-girth (note that for $t$-arc-transitive graphs
with diameter at least $t$, the odd-girth is at least $2t+1$).

\begin{proposition}\label{prop: bipartitedouble} Let $\Gamma$ be a $t$-walk-regular graph with odd-girth $2s+1$. Then
the bipartite double of $\Gamma$ is $\min\{s,t\}$-walk-regular.
\end{proposition}
\begin{proof} Let $r=\min\{s,t\}$, and let $N$ be the adjacency matrix of $\G$, with minimal idempotents $E_i$
($i=0,\dots,d$). Then the bipartite double $\widetilde{\G}$ of $\G$ has adjacency matrix $$A=\begin{bmatrix}O & N\\
N & O\end{bmatrix}$$ with (not necessarily minimal) idempotents
$$
\frac12\begin{bmatrix}E_i & \pm E_i\\
\pm E_i & E_i\end{bmatrix}\quad(i=0,\dots,d).
$$ Combinatorially it means that for every vertex $x$ of $\G$, there are two
vertices, $x^+, x^-$, of $\widetilde{\G}$ (one in each of the stable sets) and $x^+$ is adjacent to $y^-$ in
$\widetilde{\G}$ if and only if $x$ is adjacent to $y$ in $\G$.

We will now show that each distance matrix of $\widetilde{\G}$ can be expressed in terms of the corresponding distance
matrix of $\G$, at least up to distance $r$. Indeed, first of all, if $\ell$ is even and $\dist_{\G}(x,y)=\ell$, then
also $\dist_{\widetilde{\G}}(x^{\ep},y^{\ep})=\ell$, where the symbol $\ep$ can be $+$ or $-$. If moreover $\ell \leq
r$, then $\dist_{\widetilde{\G}}(x^{\ep},y^{-\ep}) \geq 2s+1-\ell \geq r+1$ (the first inequality follows because
otherwise there would be a walk between $x$ and $y$ in $\G$ of odd length less than $2s+1-\ell$, which together with
the walk of even length $\ell$ would give a closed walk of odd length less than $2s+1$, a contradiction). Similarly, it
follows that if $\ell$ is odd and $\dist_{\G}(x,y)=\ell$, then also $\dist_{\widetilde{\G}}(x^{\ep},y^{-\ep})=\ell$,
and if moreover $\ell \leq r$, then $\dist_{\widetilde{\G}}(x^{\ep},y^{\ep}) \geq r+1$. Altogether this shows that for
$\ell \leq r$ we have that
$$
A_{\ell}=\begin{bmatrix}O & N_{\ell}\\
N_{\ell} & O\end{bmatrix} \quad(\ell \text{ odd) \quad and \quad} A_{\ell}=\begin{bmatrix}N_{\ell} & O\\
O & N_{\ell} \end{bmatrix}\quad (\ell \text{ even)},
$$
where $N_{\ell}$ is the distance-$\ell$ matrix of $\G$. By taking entrywise products of these matrices with the
idempotents, the result now follows.
\end{proof}

The graph on the flags of the $11$-point biplane as described by Dalf\'o, Van Dam, Fiol, Garriga, and Gorissen
\cite{DvDFGG11} and characterized by Blokhuis and Brouwer \cite{BB12} (see also Figure \ref{fig: F11 flag graph}) is
$3$-walk-regular with odd-girth $5$, so its bipartite double is $2$-walk-regular (and it is not $3$-walk-regular).
Proposition \ref{prop: bipartitedouble} also shows that the bipartite double of the dodecahedral graph is
$2$-walk-regular because the dodecahedral graph is $5$-walk-regular with odd-girth $5$. This bipartite double is even
$3$-walk-regular, however.

\begin{figure}
\begin{center}
\includegraphics[width=100mm]{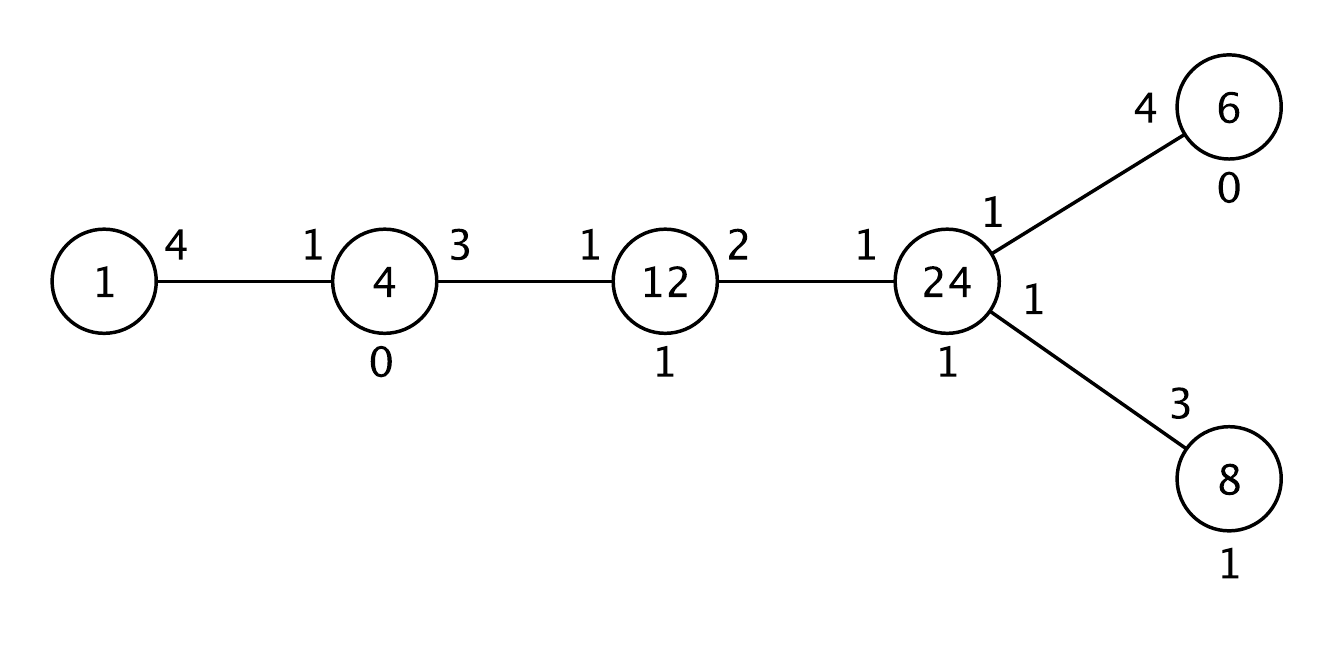}
\caption{Distance distribution diagram of a graph on the flags of the eleven point biplane}
\label{fig: F11 flag graph}
\end{center}
\end{figure}

The following result is in some sense going in the opposite direction.

\begin{proposition}\label{prop: distance2graph}
Let $t \geq 2$, and let $\Gamma$ be a $t$-walk-regular graph with valency $k$ and odd-girth $2s+1$. If $\G$ is not
complete multipartite, then the distance-$2$ graph $\G_2$ of $\G$ is $\min\{\lfloor s/2 \rfloor,\lfloor t/2
\rfloor\}$-walk-regular.
\end{proposition}
\begin{proof}
Let $r=\min\{\lfloor s/2 \rfloor,\lfloor t/2 \rfloor\}$, and let $A$ be the adjacency matrix of $\G$. Because $\G$ is
$2$-walk-regular, the adjacency matrix $B=A_2$ of $\G_2$ satisfies $$B=\frac1{c_2}(A^2-a_1A-kI)$$ and hence it has the
minimal idempotents of $A$ as (not necessarily minimal) idempotents. This, together with the claim that $B_{i}=A_{2i}$
for $i=0,\dots,r$ proves the desired result, as long as $\G_2$ is connected.

In order to prove this claim, consider two arbitrary vertices $x$ and $y$ of $\G$. If $\dist_{\G}(x,y)=2j \leq 2r$,
then $\dist_{\G_2}(x,y)=j$. If $\dist_{\G}(x,y)=2j+1<2r+1$, then $\dist_{\G_2}(x,y) \geq 2r-j>r$ (where the first
inequality follows because otherwise this would give a walk in $\G$ between $x$ and $y$ of even length less than
$2(2r-j)$, which together with the odd walk of length $2j+1$ gives a closed walk of odd length less than $4r+1$, a
contradiction). Finally, also if $\dist_{\G}(x,y) \geq 2r+1$, then $\dist_{\G_2}(x,y) >r$, which shows the claim.

The assumption that $\G$ is not complete multipartite nor bipartite (because it has finite odd-girth) ensures that
$\G_2$ is connected. Indeed, if $\G_2$ is not connected, then by using that the largest eigenvalue of
$B=\frac1{c_2}(A^2-a_1A-kI)$ occurs with multiplicity at least two, we find that $a_1-k$ must be an eigenvalue of $\G$.
Because $\G$ is not bipartite, this implies that $a_1 \neq 0$. We may also assume that $a_2=0$, because otherwise
$p^1_{22} \neq 0$, which implies that $\G_2$ is connected. Thus, if the diameter of $\G$ is two, then $\G$ is complete
multipartite. But if the diameter is at least three, then consider vertices $x$ and $y_i$ $(i=1,2,3)$ such that
$y_i\in\G_i(x)$ and $y_1\sim y_2\sim y_3$. Then
$0<a_1=|\G(y_2)\cap\G(y_3)|\leq|\G_2(x)\cap\G(y_2)|+|\G_2(y_1)\cap\G(y_3)|=2a_2=0$ (cf.~\cite[Prop.~5.5.1]{bcn89}), a
contradiction.
\end{proof}

For example, the distance-2 graph of the dodecahedral graph is $1$-walk-regular but not $2$-walk-regular. Another
example comes from the Biggs-Smith graph, whose distance-2 graph is 2-walk-regular.

We remark that the halved graphs of a bipartite graph are degenerate cases of the distance-2 graph. We thus obtain the
following.

\begin{corollary}\label{cor: halvedgraph}
Let $t \geq 2$, and let $\Gamma$ be a $t$-walk-regular bipartite graph. Then the halved graphs of $\G$ are $\lfloor t/2
\rfloor$-walk-regular.
\end{corollary}

Using that the minimal idempotents of the line graph of a regular graph are easily deduced from the minimal idempotents
of the graph, we obtain the following.

\begin{proposition}\label{prop: line graph t-w-r}
Let $t\geq 0$. Let $\Gamma$ be a $(t+1)$-walk-regular graph with valency $k$ and girth larger than
$2t+1$. Then the line graph of $\Gamma$ is $t$-walk-regular.
\end{proposition}
\begin{proof}
Let $k=\theta_0>\cdots >\theta_d$ be the distinct eigenvalues of $\G$. Consider an eigenvalue $\theta_i\neq -k$ and let
$E_i$ be the minimal idempotent associated to $\theta_i$. Let $N$ be the vertex-edge incidence matrix of $\Gamma$, then
$A=NN^{\top}-kI$ is the adjacency matrix of $\G$, and $B=N^{\top}N-2I$ is the adjacency matrix of the line graph
$L(\G)$ of $\G$. It is well-known and easy to check that if $u$ is an eigenvector of $A$ with eigenvalue $\theta_i\neq
-k$, then $N^{\top}u$ is an eigenvector of $B$ with eigenvalue $\sigma_i=\theta_i+k-2$. It follows that
$F_{\sigma_i}=\frac1{k+\theta_i}N^{\top} E_i N$ is a minimal idempotent for $L(\G)$ with
corresponding eigenvalue $\sigma_i$. Moreover, if $-2$ is an eigenvalue of $L(\G)$, then
\begin{equation}
\label{F-2}
F_{-2}=I-\sum_{\sigma_i\neq -2} F_{\sigma_i}
\end{equation}
completes the set of all minimal idempotents of $L(\G)$.

Let $e_1=uv$ and $e_2=xy$ be two edges in $\Gamma$, that is, two vertices in $L(\G)$. Assume that
$\dist_{L(\G)}(e_1,e_2)=s\leq t$, then, as the girth is larger than $2s+1$, we can assume without loss of generality
that $\dist_{\G}(u,y)=s+1$, $\dist_{\G}(u,x)=\dist_{\G}(v,y)=s$, and $\dist_{\G}(v,x)=s-1$ (except for the case $s=0$;
then $\dist_{\G}(v,x)=1$). Let $\sigma_i\neq -2$, $F=F_{\sigma_i}$, and $E=E_i$, then the $(e_1,e_2)$-entry of $F$ will
be $F_{e_1e_2}=\frac1{k+\theta_i}(E_{ux}+E_{uy}+E_{vx}+E_{vy}),$ which does not depend on the
chosen $e_1$ and $e_2$, but only on the distance $s$ between them in $L(\G)$. As this holds for every $\sigma_i \neq
-2$, by \eqref{F-2} also the entries of $F_{-2}$ depend only on distance. So we conclude that $L(G)$ is
$t$-walk-regular.
\end{proof}

An example is the already mentioned graph on the flags of the $11$-point biplane whose distance distribution diagram is
in Figure~\ref{fig: F11 flag graph}. Since this graph has girth $5$ and it is $3$-walk-regular (and therefore
$2$-walk-regular), its line graph is $1$-walk-regular (and not $2$-walk-regular). This shows that the condition on the
girth is necessary. Also the line graphs of $s$-arc-transitive graphs (with large girth) provide new examples of
$t$-walk-regular graphs. Note by the way that the line graph of a $(t+1)$-arc-transitive graph with valency at least 3
is not $t$-arc-transitive (for $t \geq 2$), since it has triangles.

We will proceed with a straightforward construction method for $1$-walk-regular graphs. Let us first recall the
coclique extension of a graph $\G$, that is, the graph with adjacency matrix $A \otimes J$, where $A$ is the adjacency
matrix of $\G$, $J$ is a square all-ones matrix and $\otimes$ stands for the Kronecker product. It is fairly easy to
see (combinatorially) that if $\G$ is a $1$-walk-regular graph, then also every coclique extension of $\G$ is
$1$-walk-regular. A variation on the coclique extension is the Kronecker product $\G \otimes \G'$ of two graphs $\G$
and $\G'$, that is, the graph with adjacency matrix $A \otimes B$, where $A$ and $B$ are the adjacency matrices of $\G$
and $\G'$.

\begin{proposition}\label{prop: kron}
Let $\G$ and $\G'$ be two $1$-walk-regular graphs. Then the Kronecker product $\G \otimes \G'$ is $1$-walk-regular.
\end{proposition}
\begin{proof}
Let $A$ and $B$ be the adjacency matrices of $\G$ and $\G'$, and let $E_i$  and $F_j$ be the respective minimal
idempotents, say with $AE_i=\theta_iE_i$, $A \circ E_i = \gamma_i A$ ($i=0,\dots,d$), $BF_j=\theta_j'F_j$, and $B
\circ F_j = \beta_j B$ ($j=0,\dots,d'$). It follows from standard multiplication rules of the Kronecker product that
$E_i \otimes F_j$ ($i=0,\dots,d; j=0,\dots,d'$) are idempotents of $\G \otimes \G'$, with eigenvalues $\theta_i
\theta_j'$ (see \cite{Grah:81}). We remark that these are not necessarily the minimal idempotents, because some
eigenvalues $\theta_i \theta_j'$ may coincide. If the latter is the case, then the corresponding minimal idempotent is
a sum of idempotents of the form $E_i \otimes F_j$.

First of all, $E_i \otimes F_j$ has constant diagonal because both $\G$ and $\G'$ are walk-regular, and so every
minimal idempotent has constant diagonal. So $\G \otimes \G'$ is also walk-regular.

Secondly, it follows that $(A \otimes B) \circ (E_i \otimes F_j) = (A \circ E_i) \otimes (B \circ F_j) =
\gamma_i\beta_j A \otimes B$, which shows that $\G \otimes \G'$ is 1-walk-regular.
\end{proof}

Note that it was already observed by Godsil and McKay \cite[Thm.~4.5]{Godsil198051} that several kinds of products,
such as the Kronecker product and the Cartesian product of walk-regular graphs are again walk-regular. Still, the
Cartesian product (or sum \cite{cds82}) $\G \oplus\G'$ of two $1$-walk-regular graphs $\G$ and $\G'$, that is, the
graph with adjacency matrix $A \otimes I +I \otimes B$, is in general not $1$-walk-regular. However, the particular
case $\G \oplus \G$ is again $1$-walk-regular, as one can easily show (the idempotents are $E_i \otimes E_j + E_j
\otimes E_i$ ($i \neq j$) and $E_i\otimes E_i$).

As announced in the Introduction of this paper, we will observe a structural gap between $1$-  and $2$-walk-regular
graphs. Clearly, we cannot speak of such a gap for graphs with valency $2$, because all such (connected) graphs are
distance-regular. The final result of this section is that also valency $3$ is special, in the sense that cubic
$1$-walk-regular graphs are also $2$-walk-regular. In particular, every cubic $1$-arc-transitive graph
is $2$-walk-regular.

\begin{proposition}\label{proporg: cubic 1-walk} Let $\Gamma$ be a cubic $1$-walk regular graph with diameter at least two. Then
$\Gamma$ is $2$-walk regular.
\end{proposition}
\begin{proof}
Let $\theta$ be an eigenvalue of $\G$ and consider the cosines with respect to this eigenvalue. Let $u_1$ be the cosine
for two vertices at distance $1$. Let $x$ and $y$ be two vertices at distance $2$. Consider a common neighbor $z$ of
$x$ and $y$, and let $w$ be the third neighbor of $z$.  By considering the $(z,x)$-entry in the equation
$AE_{\theta}=\theta E_\theta$, we obtain that $1 + u_{yx} + u_{wx}=\theta u_1$. Similarly, we find that $1 + u_{xy} +
u_{wy}=\theta u_1$ and $1 + u_{xw} + u_{yw}=\theta u_1$. These three equations imply that $u_{xy} = (\theta u_1 -1)/2$.
So the cosine for two vertices at distance $2$ is constant (for every eigenvalue), hence $\G$ is $2$-walk-regular.
\end{proof}

\section{Godsil's multiplicity bound}\label{sec:godsil}
Let $m \geq 2$ and let $\Gamma$ be a connected regular graph with an eigenvalue $\theta\neq \pm k$ with multiplicity
$m$. Godsil \cite{g88} proved that if such a graph is distance-regular and not complete multipartite, then both its
diameter and its valency are bounded by a function of $m$. In particular, this assures that there are finitely many
such distance-regular graphs. In this section we extend some of Godsil's results and reasonings to the class of
$2$-walk-regular graphs. The main difference with distance-regular graphs is that we are not able to bound the
diameter.

We start by pointing out that, as it happens with distance-regular graphs, the images of two vertices at distance at
most $2$ under a representation associated to $\theta\neq\pm k$ cannot be collinear. The following lemma can indeed be
read between the lines in a proof by Godsil \cite{g88}.

\begin{lemma}\label{lem: 2-walk x=y}
Let $\Gamma$ be a $2$-walk-regular graph different from a complete multipartite graph, with valency $k\geq 3$ and
eigenvalue $\theta\neq \pm k$. Let $x$ and $y$ be vertices of $\Gamma$ and consider a representation associated to a
$2$-walk-regular idempotent for $\theta$. If $\hat{x}=\pm\hat{y}$, then $\dist(x,y)>2$.
\end{lemma}
\begin{proof}
Assume that $x,y\in V$ are such that $\hat{x}=\pm \hat{y}$. Then $u_{xy}=\langle \hat{x},\hat{y}\rangle /\alpha_0=\pm 1$.
If $x$ and $y$ are adjacent, then this implies that $u_1=u_{xy}=\pm 1$. From \eqref{eq: relation alpha0 alpha1} we find
that $u_1=\theta/k$, so $\theta=\pm k$, a contradiction. Suppose now that $\dist(x,y)=2$. Then $u_2=\pm 1$, so for
every pair of vertices $v$ and $w$ at distance $2$, we have that $\hat{v}=\pm \hat{w}$.

Assume first that $\Gamma$ is triangle free.  Let $z$ be a common neighbor of $x$ and $y$. Then, as $k\geq 3$, we must
have $\hat{x}=\hat{y}$ (otherwise $u_2=-1$, so if $w$ is another neighbor of $z$, then
$\hat{x}=-\hat{y}=\hat{w}=-\hat{x}$, a contradiction). Then \eqref{eq: representation lambda} gives $\theta
\hat{z}=\sum_{w\sim z}\hat{w}=k\hat{x}$. Since $\|\hat{x}\|=\|\hat{z}\|$, this implies that $\theta=\pm k$, again a
contradiction.

Suppose now that $\Gamma$ is not triangle free. As $\G$ is $2$-walk-regular, every edge is in a triangle. Let $v$ and
$w$ be two vertices at distance $2$. Let $z_1$ be a common neighbor of $v$ and $w$ and let $z_2$ be a common neighbor
of $v$ and $z_1$. Then $\hat{v}=\pm\hat{w}$. If $w$ is not adjacent to $z_2$, then $\hat{v}=\pm\hat{w}=\pm\hat{z_2}$
and $u_1=\pm 1$, a contradiction. So $w$ is adjacent to $z_2$. By the same argument, every other neighbor $z_3$ of $v$
is adjacent to $z_1$ or $z_2$, and hence to $w$. So every two vertices at distance $2$ have $k$ common neighbors, so
$\G$ is complete multipartite, a final contradiction.
\end{proof}

An immediate corollary is the following.

\begin{corollary}
Let $\Gamma$ be a $2$-walk-regular graph different from a complete multipartite graph, with valency $k\geq 3$ and
eigenvalue $\theta\neq k$, and consider the representation associated to a $2$-walk-regular idempotent
for $\theta$. If $u_2=\pm 1$, then $\theta=-k$ and $\Gamma$ is bipartite.
\end{corollary}

Let $\theta\neq \pm k$ be an eigenvalue of a 2-walk-regular graph $\G$ (with valency $k$) and consider the (spherical)
representation associated to a $2$-walk-regular idempotent with rank $m$ for $\theta$. Let $x$ be a
vertex of $\G$ and consider the set of vectors $\{\hat{y}\,|\, y\in \Gamma(x)\}$. These vectors lie in the hyperplane
of all vectors having inner product $\alpha_1$ with $\hat{x}$, so they lie in an $(m-1)$-dimensional sphere (in
$\R^m$). Lemma~\ref{lem: 2-walk x=y} ensures that the cardinality of the set is $k$. Also, the inner product between
two of its elements is either $\alpha_1$ or $\alpha_2$, so it is a (spherical) $2$-distance set. As pointed out by
Godsil \cite[Lemma 4.1]{g88}, Delsarte, Goethals, and Seidel \cite[Ex.~4.10]{DGS77} provide a bound for the size of
such a set, and we have the following (cf.~\cite[Thm.~1.1]{g88}):

\begin{theorem}\label{thm: bound k with m}
Let $\Gamma$ be a $2$-walk-regular graph, not complete multipartite, with valency $k\geq 3$. Assume that $\Gamma$ has
a $2$-walk-regular idempotent with rank $m$ for an eigenvalue $\theta\neq \pm k$. Then $k\leq
\frac{(m+2)(m-1)}{2}.$
\end{theorem}

The assumption in this result is of course satisfied if $\G$ has an eigenvalue with multiplicity
$m>1$. The obtained bound will be key in Section~\ref{sec: geometric}, as well as for the study of $2$-walk-regular
graphs with an eigenvalue with multiplicity $3$ in Section~\ref{subsec: 2-walk-reg small mult}. In both cases we will
also use properties of the local graph of $2$-walk-regular graphs; we will study these in the next section. Note that
also some of the results in Terwilliger's `tree bound' paper \cite{T82} on $t$-arc-transitive graphs and in Hiraki and
Koolen's paper \cite{HirakiK} with improvements of Godsil's bound can be generalized to $t$-walk-regular graphs with
large enough girth.

\section{The local structure of 2-walk-regular graphs}\label{sec:terwilliger}
In \cite{T86} Terwilliger gave bounds for the eigenvalues of the local graphs of a distance-regular graph (see also 
\cite[Thm.~4.4.3]{bcn89} and \cite[Cor.~4.3, Ch. 13]{g93b}). We start
this section showing that these bounds also hold for 2-walk-regular graphs. We follow the proof as given by Godsil
\cite[Ch.~13]{g93b}.

\begin{proposition}\label{prop: local 1} Let $\Gamma$ be a
$2$-walk-regular graph with distinct eigenvalues $k=\theta_0>\cdots>\theta_d$. Let $x$ be a vertex of $\Gamma$
and let $\Delta$ be the subgraph of $\Gamma$ induced on the neighbors of $x$. Let
$a_1=\eta_0\geq\cdots\geq\eta_{k-1}$ be the eigenvalues of $\Delta$. Then
\begin{align*}
 \eta_{k-1} \geq & -1-\frac{b_1}{\theta_1+1},\\
 \eta_1 \leq & -1 -\frac{b_1}{\theta_d+1}.
\end{align*}
\end{proposition}
\begin{proof}
Let $\theta\neq k$ be an eigenvalue of $\Gamma$ and let $E:=E_{\theta}$ be the minimal idempotent corresponding to
$\theta$. Since $\Gamma$ is $2$-walk-regular, the intersection numbers $a_j$, $b_j$ and $c_j$ $(j=0,1,2)$ are
well-defined and there exist constants $\alpha_i$ for $i\in\{0,1,2\}$ such that $E\circ A_i=\alpha_i A_i$, where
$\circ$ is the entrywise product. Also, \eqref{eq: relation alpha0 alpha1} and \eqref{eq: relations alpha's} lead to
 $\alpha_1=\alpha_0\frac{\theta}{k}$ and
$\alpha_2=\frac{1}{b_1}((\theta-a_1)\alpha_1-\alpha_0)=\alpha_0(\frac{\theta^2}{b_1k}-\frac{a_1\theta}{b_1k}-\frac{1}{b_1})$.

Let $E_{\Delta}$ be the principal submatrix of $E$ on the vertices of $\Delta$. Clearly, $E_{\Delta}$ is positive
semidefinite and has the form $E_{\Delta}=\alpha_0I+\alpha_1A(\Delta)+\alpha_2(J-I-A(\Delta))$, where $J$ is the
all-ones matrix and $A(\Delta)$ is the adjacency matrix of $\Delta$.

Let $w$ be an eigenvector of $A(\Delta)$ with corresponding eigenvalue $\eta$ that is orthogonal to the all-ones
vector. Then $E_{\Delta}w=(\alpha_0+\alpha_1\eta+\alpha_2(-1-\eta))w$, which implies that
$\alpha_0+\alpha_1\eta+\alpha_2(-1-\eta)\geq0$ as $E_{\Delta}$ is positive-semidefinite. Since $\alpha_0>0$,
$\alpha_1=\alpha_0\frac{\theta}{k}$, and
$\alpha_2=\alpha_0(\frac{\theta^2}{b_1k}-\frac{a_1\theta}{b_1k}-\frac{1}{b_1})$, we find that
$(\theta-k)((1+\eta)\theta-(a_1-k-\eta))\leq0$ and this shows that $(1+\eta)\theta-(a_1-k-\eta)\geq0$ as $\theta< k$.
So we have $\eta(\theta+1)\geq-(\theta+1)-b_1$. This completes the proof.
\end{proof}

We remark that the $2$-coclique extensions of the lattice graphs $L_2(n)$ provide examples of $1$-walk-regular graphs
for which the upper bound for the eigenvalues of the local graphs in the above proposition is not valid. In this case
$\eta_1=a_1=2n-4$ (the local graph consists of $2$ cocktailparty graphs), $b_1=2n-1$, and $\theta_d=-4$.

In what follows the symbol $\delta_{x,y}$ stands for the Kronecker delta, that is, $\delta_{x,y}=1$ if $x=y$ and $0$
otherwise.

\begin{proposition}\label{thm: lacal 2} 
Let $\Gamma$ be a $2$-walk-regular graph with distinct eigenvalues $k=\theta_0>\cdots>\theta_d$. Let $x$ be a
vertex of $\Gamma$ and let $\Delta$ be the subgraph of $\Gamma$ induced on the neighbors of $x$. Let $E$ be
a $2$-walk-regular idempotent with rank $m$ for an eigenvalue $\theta\neq \pm k$. If $m<k$, then
$\theta\in\{\theta_1,\theta_d\}$ and $b:= -1-\frac{b_1}{\theta+1}$ is an eigenvalue of $\Delta$ with multiplicity at
least $k-m+ \delta_{b, a_1}$.
\end{proposition}
\begin{proof}
Let $E_{\Delta}$ be the principal submatrix of $E$ indexed by the vertices of $\Delta$. Clearly,
$E_{\Delta}=\alpha_0I+\alpha_1A(\Delta)+\alpha_2(J-I-A(\Delta))$.

Now it follows first of all that $\theta\neq-1$, because if $\theta=-1$, then \eqref{eq: relation
alpha0 alpha1} and \eqref{eq: relations alpha's} imply that $k\alpha_1=-\alpha_0$ and
$\alpha_0+\alpha_1a_1+\alpha_2(k-1-a_1)=-\alpha_1$, and from this it follows that $\alpha_2=\alpha_1= -\alpha_0/k$.
This implies that $m =\rk(E) \geq \rk(E_{\Delta})=\rk(\alpha_0I+\alpha_1(J-I))=k$, which is indeed a contradiction.

Next, we obtain that $0$ is an eigenvalue of $E_{\Delta}$ with multiplicity at least $k-m$ as $\rk(E_{\Delta})\leq
\rk(E)=m<k$ and $E_{\Delta}$ is a $k\times k$ matrix. Let us the consider the possible eigenvectors for this
eigenvalue.

If $w$ is an eigenvector of $A(\Delta)$ orthogonal to the all-ones vector, with corresponding eigenvalue $\eta$, then
$w$ is an eigenvector of $E_{\Delta}$ with eigenvalue $\alpha_0+\alpha_1\eta+\alpha_2(-1-\eta)$. If the latter
eigenvalue equals $0$, then by a derivation similar to that in the proof of Proposition \ref{prop: local 1} and using
that $\theta \neq -1$, it follows that $\eta=-1-\frac{b_1}{\theta+1}$. By Proposition \ref{prop: local 1}, we have that
$\eta \leq -1-\frac{b_1}{\theta_{d}+1}$, so if $\theta<-1$, then $\theta=\theta_{d}$. Similarly, it follows that if
$\theta>-1$, then $\theta=\theta_1$.

What remains is to check the all-ones eigenvector. The corresponding eigenvalue of $E_{\Delta}$ is equal to
$\alpha_0+\alpha_1a_1+\alpha_2(k-1-a_1)=\theta\alpha_1=(\theta^2/k) \alpha_0$, where the two equalities follow as
before. Because $\alpha_0 \neq 0$, it follows that if this eigenvalue of $E_{\Delta}$ is $0$ then $\theta = 0$. Because
$\theta_1$ and $\theta_d$ are both nonzero, the above shows that in this case $\rk(E_{\Delta})\geq k-1$. Now consider
the subgraph $\overline{\Delta}$ of $\Gamma$ induced on $x$ and its neighbors. Because
$\alpha_1=\alpha_0\frac{\theta}{k}=0$, the corresponding submatrix $E_{\overline{\Delta}}$ of $E$ has rank
$\rk(E_{\Delta})+1$, which is at least $k$, and this contradicts $m<k$.

It thus follows that $\theta\in\{\theta_1,\theta_{d}\}$ and that  $b=-1-\frac{b_1}{1+\theta}$ is an eigenvalue of
$\Delta$ with multiplicity at least $k-m$. Moreover, if $b=a_1$, then the all-ones vector is also an eigenvector for
eigenvalue $b$, so that the multiplicity is at least $k-m+ \delta_{b, a_1}$.
\end{proof}

By taking for the matrix $E$ the minimal idempotent corresponding to an eigenvalue of $\Gamma$ we obtain (cf.~\cite[Thm.~4.4.4]{bcn89} and \cite[Thm.~4.2, Ch. 13]{g93b}):
\begin{corollary}\label{thm: lacal' 2} Let $\Gamma$ be a $2$-walk-regular graph with distinct
eigenvalues $k=\theta_0>\cdots>\theta_d$ and local graph $\Delta$. Let $\theta \neq k$ be an eigenvalue of
$\Gamma$ with multiplicity $m$. If $m<k$, then $\theta\in\{\theta_1,\theta_d\}$ and $b:= -1-\frac{b_1}{\theta+1}$ is an
eigenvalue of $\Delta$ with multiplicity at least $k-m+ \delta_{b, a_1}$.
\end{corollary}

We remark that instead of taking the local graph $\Delta$, we may take any regular induced subgraph $\Sigma$ with the
property that every two distinct non-adjacent vertices in $\Sigma$ have distance $2$ in $\Gamma$. See also Koolen
\cite{K94}.

In the next part we are going to derive the `fundamental bound' for 2-walk-regular graphs. This bound was obtained for
distance-regular graphs by Juri\v{s}i\'{c}, Koolen, and Terwilliger \cite{JKTtight}. We follow the proof of this bound
as given by Juri\v{s}i\'{c} and Koolen \cite{JK00} and start with the following lemma. For the convenience of the
reader we provide a proof of it.

\begin{lemma}\label{lem: Jurisic and Koolen} \cite[Thm.~2.1]{JK00}
Let $\Delta$ be a regular graph with valency $k$ and $n$ vertices. Let $k= \eta_0  \geq \cdots \geq
\eta_{n-1}$ be the eigenvalues of $\Delta$. Let $\sigma$ and $\tau$  be numbers such that $\sigma \geq \eta_1 \geq
\eta_{n-1} \geq	\tau$. Then $n(k+ \sigma\tau) \leq (k- \sigma)(k-\tau),$ with equality if and only if
$\eta_i\in\{\sigma,\tau\}$ $(1\leq i\leq n-1)$. In particular, if equality holds then $\Delta$ is empty, complete, or
strongly regular.
\end{lemma}
\begin{proof}
Note that by assumption of the lemma we have
\begin{equation}
\label{eq:fundmental}
\sum_{i=1}^{n-1}(\eta_i-\sigma)(\eta_i-\tau)\leq0.
\end{equation}
As $\displaystyle\sum_{i=0}^{n-1}\eta_i=0$, $\displaystyle\sum_{i=0}^{n-1}\eta_i^2=nk$ and $\eta_0=k$, the inequality
in the lemma immediately follows.

In case of equality we obtain equality in \eqref{eq:fundmental} which in turn implies that $\eta_i\in\{\sigma,\tau\}$
$(1\leq i\leq n-1)$.
\end{proof}

As a consequence of Proposition~\ref{prop: local 1} and Lemma~\ref{lem: Jurisic and Koolen} we obtain the following
`fundamental bound' (cf.~\cite[Thm.~6.2]{JKTtight} and \cite[Thm.~2.1]{JK00}).

\begin{theorem} 
Let $\Gamma$ be a $2$-walk-regular graph with distinct eigenvalues $k=\theta_0>\cdots>\theta_d$. Then
$$(\theta_1 +\frac{k}{ a_1 +1}  )(\theta_d +\frac{k}{ a_1 +1 } )\geq -\frac{k a_1 b_1}{  (a_1 +1)^2}.$$
If $a_1 \neq 0$, then equality holds if and only if every local graph $\Delta$ is strongly regular with eigenvalues
$a_1$, $-1-\frac{b_1}{\theta_d+1}$, and $-1-\frac{b_1}{\theta_1+1}$. If $a_1=0$, then equality holds if and only if
$\Gamma$ is bipartite.
\end{theorem}
\begin{proof}
Let $x$ be a vertex of $\Gamma$ and let $\Delta:=\Delta(x)$ be the subgraph of $\Gamma$ induced on the neighbors of
$x$. Let $a_1=\eta_0\geq\cdots\geq\eta_{k-1}$ be the eigenvalues of $\Delta$. Let
$\sigma=-1-\frac{b_1}{\theta_d+1}$ and $\tau=-1-\frac{b_1}{\theta_1+1}$. Then we have
$\sigma\geq\eta_1\geq\eta_{k-1}\geq\tau$ by Proposition~\ref{prop: local 1}. As $\Delta$ is a regular graph with
valency $a_1$ and $k$ vertices, we obtain the fundamental bound by reformulating the inequality in Lemma~\ref{lem:
Jurisic and Koolen}; we omit the technical details. If $a_1=0$ (and the local graph is empty), then equality holds if
and only if $\theta_d=-k$.
\end{proof}

\section{Small multiplicity}\label{sec:smallmult}
This section is devoted to study $t$-walk-regular graphs having eigenvalues with small multiplicity. We start by
answering the following question: How small can the multiplicity of an eigenvalue be of a $t$-walk-regular graph that
is not distance-regular? Afterwards, in Sections \ref{subsec: 1-walk-reg small mult} and \ref{subsec: 2-walk-reg small
mult}, we will use this answer and the results in the previous sections to describe $1$- and $2$-walk-regular graphs
having an eigenvalue (with absolute value smaller than the spectral radius) with small multiplicity.

\subsection{Distance-regularity from a small multiplicity}
Dalf\'o, Van Dam, Fiol, Garriga and Gorissen \cite{DvDFGG11} posed the following problem: What is the smallest $t$ such
that every $t$-walk-regular graph is distance-regular? More precisely, they considered $t$ as a function of either the
diameter $D$ of $\G$ or the number $d+1$ of distinct eigenvalues. We will give an answer to this question, but in terms
of the minimum multiplicity of an eigenvalue $\theta\neq \pm k$ of $\G$ (where $k$ is the valency), or actually a bit stronger, in terms of the minimum rank of a $t$-walk-regular idempotent for $\theta$. Notice that the
minimum multiplicity is related to $d$ and the number of vertices. The following result follows from revisiting the
proof of a result by Godsil \cite[Thm.~3.2]{g88}.

\begin{proposition}\label{prop: m-walk-regular mult m b_m=1}
Let $t\geq 2$ and let $\Gamma$ be a $t$-walk-regular graph with valency $k \geq 3$ and diameter $D>t$. If $\G$ has a $t$-walk-regular idempotent for an eigenvalue $\theta\neq \pm k$ with rank at most $t$, then $b_t=1$.
\end{proposition}
\begin{proof}
Consider the representation associated to the $t$-walk-regular idempotent $E$ for $\theta$. Let $x=x_0$ and $y=x_t$ be two vertices at distance $t$ in $\Gamma$, and let $P=x_0\ldots x_{t}$ be a (shortest)
path joining them. Let $Q=x_0\ldots x_q$ be the longest subpath of $P$ starting at $x$ such that
$\{\hat{x}_i\}_{0\leq i\leq q}$ are linearly independent (clearly $q+1\leq t$, since the maximum number of linearly
independent $\hat{x}_i$ is at most the rank of $E$). Therefore,
$\hat{x}_{q+1}=p_0\hat{x}_0+p_1\hat{x}_1+\cdots+p_q\hat{x}_q$, for certain coefficients $p_i$. If $z$ is a vertex
adjacent to $x_{q+1}$ that is at distance $q+2$ from $x$, then we claim that
$\hat{z}=p_0\hat{x}_1+p_1\hat{x}_2+\cdots+p_q\hat{x}_{q+1}$. Indeed,
$$
0=\| x_{q+1}-(p_0\hat{x}_0+p_1\hat{x}_1+\cdots+p_q\hat{x}_q)\|=\|
z-(p_0\hat{x}_1+p_1\hat{x}_2+\cdots+p_q\hat{x}_{q+1})\|,
$$
which holds because $\langle \hat{z}, \hat{x}_{q+2-i} \rangle =\langle \hat{x}_{j+i}, \hat{x}_{j}\rangle =\alpha_i$ for
$i=1,\dots,q+1$ and $j=0,\dots,q+1-i$.

Recall that the intersection numbers $a_i$, $b_i$, and $c_i$ are well-defined for $i=0,\dots,t$. Suppose now that
$b_{q+1}>1$, and let $z_1$ and $z_2$ be two vertices adjacent to $x_{q+1}$ and at distance $q+2$ from $x$. Then by the
above, we have that $\hat{z}_1=p_0\hat{x}_1+p_1\hat{x}_2+\cdots+p_q\hat{x}_{q+1}=\hat{z}_2$. By Lemma~\ref{lem: 2-walk
x=y}, this implies that $\theta=\pm k$, a contradiction, so $b_{q+1}=1$. Now observe that in the same way as for
distance-regular graphs (see \cite[Prop.~4.1.6]{bcn89}), we have that $b_i\leq b_j$ if $j\leq i\leq t$. Therefore
$b_t=1$.
\end{proof}
\begin{proposition}\label{prop: t-wr bt=1 DRG}
Let $\Gamma$ be a $t$-walk-regular graph. If $b_t=1$, then $\Gamma$ is distance-regular.
\end{proposition}
\begin{proof}
We will show that if $t<D$ and $b_t=1$, then $\Gamma$ is also $(t+1)$-walk-regular. Since $b_{t+1}\leq b_t$
(cf.~\cite[Prop.~4.1.6]{bcn89}), the statement then follows by induction.

Let $x$ and $z$ be vertices at distance $t+1$, and let $y$ be a neighbor of $z$ at distance $t$ from $x$. Because
$b_t=1$, the only neighbor of $y$ at distance $t+1$ from $x$ is $z$. Let $E$ be the minimal idempotent of an eigenvalue
$\theta$ of $\Gamma$. By considering the $(x,y)$-entry of $AE=\theta E$, we thus obtain that
$c_t\alpha_{t-1}+a_t\alpha_t+E_{xz}=\theta\alpha_t.$ This shows that $E_{xz}$ does not depend on $x$ and $z$, but only
on their distance $t+1$. Therefore $\Gamma$ is $(t+1)$-walk-regular.
\end{proof}

We remark that the second part of the proof generalizes, in the sense that it actually proves that if $E$ is a
$t$-walk-regular idempotent for an eigenvalue in a $t$-walk-regular graph with $b_t=1$, then $E$ is a
$(t+1)$-walk-regular idempotent. The following result now follows immediately.

\begin{theorem} \label{thm: m-walk multiplicity m implies DRG}
Let $\Gamma$ be a $t$-walk-regular graph with a $t$-walk-regular idempotent for an eigenvalue $\theta\neq \pm k$ with rank at most $t$. If $t\geq 2$, then $\Gamma$ is distance-regular.
\end{theorem}

Let us stress once more that if $\G$ has an eigenvalue $\theta\neq \pm k$ with multiplicity at most
$t$, then it has a corresponding $t$-walk-regular idempotent for $\theta$ with rank at most $t$,
so we obtain the following result.
\begin{corollary} \label{cor: m-walk multiplicity m implies DRG} Let $\Gamma$ be a $t$-walk-regular
graph with an eigenvalue $\theta\neq \pm k$ with multiplicity at most $t$. If $t\geq 2$, then $\Gamma$ is
distance-regular.
\end{corollary}

Note that we can extend this result with $t=1$, as we will show next that $1$-walk-regular graphs with an eigenvalue
$\theta\neq \pm k$ with multiplicity $1$ do not exist.

\subsection{$1$-Walk-regular graphs with a small multiplicity}\label{subsec: 1-walk-reg small mult}
Let $\Gamma$ be a 1-walk-regular graph, and suppose that it has an eigenvalue $\theta$ with multiplicity $1$. Let $x$
and $y$ be two adjacent vertices. Since the minimal idempotent $E_{\theta}$ has rank 1, by considering the determinant
of the $2 \times 2$ principal submatrix of $E_{\theta}$ on $x$ and $y$, it follows that $\alpha_1=\pm \alpha_0$, and
hence by \eqref{eq: relation alpha0 alpha1} we obtain that $\theta=\pm k$. In other words, a $1$-walk-regular graph has
no eigenvalues different from $\pm k$ with multiplicity $1$. In the following proposition we consider $1$-walk-regular
graphs with an eigenvalue with multiplicity $2$.

\begin{proposition}\label{prop: 1-walk reg mult 2}
Let $\Gamma$ be a $1$-walk-regular graph with a $1$-walk-regular idempotent for an eigenvalue with
rank $2$. Then $\Gamma$ is a cover of a cycle.
\end{proposition}
\begin{proof}
Let $E$ be a $1$-walk-regular idempotent for an eigenvalue $\theta$ with rank $2$, then $\theta$
has multiplicity at least $2$, so $\theta \neq \pm k$. Consider the representation associated to $\theta$, and notice
that the images $\hat{x}$ under this representation are in $\R^2$, i.e., the representation is in the plane.

We may assume that $\Gamma$ is not a complete graph, because it is straightforward to derive that
the only complete graph with a $1$-walk-regular idempotent for an eigenvalue with rank 2 is the $3$-cycle. Consider
two vertices $x$ and $z$ at distance $2$ and let $y$ be a common neighbor of them. Because the determinant of the
principal submatrix of $\frac1{\alpha_0}E$ on $x,y,z$ equals $0$, and the fact that $u_{xy}=u_{yz}=u_1$, it follows
that the cosine $u_{xz}$ between two vertices at distance $2$ equals $1$ or $2u_1^2-1$ (note that the latter is the
cosine of twice the angle with cosine $u_1$).

Consider the quotient graph $\overline{\G}$ obtained by identifying vertices that have the same image under the
representation (that is, $x$ and $x'$ are identified if and only if $\hat{x}=\hat{x}'$). Let $\overline{x}$ be a class
of vertices, with $x\in\overline{x}$. Assume that the class of vertices $\overline{y}$, with $y\in \overline{y}$, is
adjacent to $\overline{x}$ in $\overline{\G}$. Then $\langle \hat{x},\hat{y}\rangle= u_1$. In the plane, this is only
possible for two different vectors $\hat{y}$, so $\overline{x}$ will be adjacent to at most two other classes.

Consider two adjacent vertices $x$ and $y$ in $\G$. Let $s$ denote the number of neighbors $z$ of $y$ that are at
distance $2$ from $x$ with $\langle x,z\rangle= 2u_1^2-1$. Then by looking at the $(y,x)$-entry of
$\frac{1}{\alpha_0}AE=\frac{1}{\alpha_0}\theta E$, we find that
$$
1+a_1u_1+s(2u_1^2-1)+(k-1-a_1-s)1=\theta u_1,
$$
so $s=\frac{k}{2}\left(1-a_1/(\theta+k)\right)$. If $a_1=0$, then $s=k/2$, and $\overline{y}$, and hence all the
vertices in $\overline{\G}$, have degree $2$, so $\overline{\G}$ is a cycle. Moreover, the partition given by the
classes of vertices is equitable (with every vertex being adjacent to $k/2$ vertices in each neighboring class).

Assume finally that $a_1> 0$. By considering the principal submatrix of $\frac{1}{\alpha_0}E$ on the three vertices of
a triangle, we find that $u_1=-1/2$ (because $u_1 \neq 1$), and hence that $2u_1^2-1=-1/2$ and $\theta=-k/2$. Let
$x\sim y$, then the common neighbors of $x$ and $y$, and the vertices $z\in \G_2(x) \cap \G(y)$ such that
$\langle\hat{x},\hat{z}\rangle=2u_1^2-1=-1/2$ are in the same class in the quotient graph. Hence, the quotient graph
$\overline{\G}$ is a triangle, and again the partition is equitable.
\end{proof}

Every coclique extension of a cycle is 1-walk-regular (see Section \ref{sec: constructions}), and it has eigenvalues
with multiplicity 2, except for coclique extensions of the $4$-cycle (which are complete bipartite graphs). But this
certainly does not cover all the possibilities.

Indeed, let $\G$ be any 1-walk-regular graph (for example, a strongly regular graph) and let $\G'$ be any cycle,
except the $4$-cycle. Then by applying Proposition~\ref{prop: kron} one obtains a 1-walk-regular graph, which typically
has an eigenvalue with multiplicity 2. Indeed, if $k$ is the valency of $\G$ and $\theta \neq 0$ is an eigenvalue of
$\G'$ with multiplicity 2, then the product $k\theta$ is a good candidate eigenvalue with multiplicity 2 of $\G \otimes
\G'$; sometimes however this eigenvalue coincides with other (product) eigenvalues. The latter clearly happens when
$\G'$ is the $4$-cycle, because its only eigenvalue with multiplicity 2 is $\theta=0$.

To show that not all $1$-walk-regular graphs with an eigenvalue with multiplicity 2 come from the above product
construction, we next present examples that have eigenvalue 0 with multiplicity 2, and the $4$-cycle as a quotient.

Consider a connected regular graph $\G$ with $n$ vertices and adjacency matrix $A$, minimal idempotents $E_0=\frac1n J,
E_1,\dots,E_d$, and  spectrum $k=\theta_0^1,\theta_1^{m_1},\dots,\theta_d^{m_d}$, where the superscripts stand for the
multiplicities. Let $\overline{A}$ be the adjacency matrix of the complement of $\G$. Then the graph $\G'$ with
adjacency matrix
$$N=\begin{bmatrix}A & \overline{A}\\
\overline{A} & A\end{bmatrix}$$ has (not necessarily minimal) idempotents
$$\frac1{2n}J, \frac12\begin{bmatrix}E_0 & -E_0\\ -E_0 &
E_0\end{bmatrix}, \frac12\begin{bmatrix}I-E_0 & I-E_0\\ I-E_0 & I-E_0\end{bmatrix}, F_i=\frac12\begin{bmatrix}E_i & -
E_i\\ - E_i & E_i\end{bmatrix} (i=1,\dots,d),$$ and (corresponding) spectrum
$n-1^1,2k-n+1^1,-1^{n-1},2\theta_1+1^{m_1},\dots,2\theta_d+1^{m_d}$. Clearly, $\G'$ is walk-regular if $\G$ is
walk-regular. If $\G$ is strongly regular with parameters $(n,k,\la,\mu)$, then one can show that $\G'$ has an
eigenvalue ($2k-n+1$) with multiplicity 1 if and only if $n \neq 4k-2\mu-2\la$. We remark that in the exceptional case
where $n=4k-2\mu-2\la$, the graph $\G$ is in the switching class of a regular two-graph, and $\G'$ is the corresponding
distance-regular Taylor graph (cf.~\cite[Thm.~1.5.6]{bcn89}). Note that this shows that there are infinitely many
walk-regular graphs with an eigenvalue (whose absolute value is smaller than the valency) with
multiplicity 1. Several such infinite families (but with only four distinct eigenvalues) were already constructed by
Van Dam \cite{vdam4}, who also studied the structure of such graphs.

Now consider the bipartite double $\widetilde{\G}'$ of $\G'$. If $\G$ is a conference graph, then the above eigenvalue
of $\G'$ with multiplicity 1 equals 0, so $\widetilde{\G}'$ has eigenvalue 0 with multiplicity 2 (its entire spectrum
consists of eigenvalues $\pm (n-1)^1,0^2,\pm 1^{n-1},\pm \sqrt{n}^{n-1}$). Moreover, $\widetilde{\G}'$ is
1-walk-regular (whereas $\G'$ is not; cf.~Proposition \ref{prop: bipartitedouble}), which follows by considering the
minimal idempotents of $\widetilde{\G}'$ (which can be obtained by summing the appropriate idempotents given in the
proof of Proposition \ref{prop: bipartitedouble}); the crucial point is that $N \circ (F_1-F_2)$ is a multiple of $N$,
but we omit details.

Thus, for every $n \geq 3$, there are infinitely many examples of 1-walk-regular graphs with an eigenvalue with
multiplicity 2 being covers of $C_n$.

We end this section by observing that the smallest multiplicity of an eigenvalue different from $k$ in a
$1$-walk-regular graph provides a bound for its clique number.

\begin{proposition}\label{prop:cliquerank}
Let $\Gamma$ be a $1$-walk-regular graph with valency $k$. Let $E$ be a $1$-walk-regular idempotent with rank $m$ for an eigenvalue $\theta\neq k$. Then every clique of $\Gamma$ has at most $m+1$ vertices.
\end{proposition}
\begin{proof}
Let $C$ be a clique with size $c$. Then the principal submatrix of $E$ indexed by the vertices of $C$ equals
$\alpha_0 I+\alpha_1 (J-I)$. Therefore, since $\alpha_0\neq \alpha_1$ (see \eqref{eq: relation alpha0 alpha1}),
$E$ has rank at least $c-1$.
\end{proof}

The coclique extensions of the triangle satisfy the bound with equality (with $m=2$), for example.

\subsection{$2$-Walk-regular graphs with a small multiplicity}\label{subsec: 2-walk-reg small mult}
Let $\theta\neq k$ be an eigenvalue of a $2$-walk-regular graph $\Gamma$ with valency $k$. Recall that $\theta$, as
proven in Section~\ref{subsec: 1-walk-reg small mult}, cannot have multiplicity one. If $\theta$ has multiplicity $2$,
then by Corollary \ref{cor: m-walk multiplicity m implies DRG} we know that $\Gamma$ is distance-regular, and the only
distance-regular graphs with an eigenvalue with multiplicity $2$ are the polygons and the regular complete tripartite
graphs (see \cite[Prop.~4.4.8]{bcn89}). In Theorem~\ref{thm: multi 3}, we will discuss the case of multiplicity $3$.
For that we use the following lemma (cf.~\cite[Thm.~1]{T82}), which is interesting on its own.

\begin{lemma}\label{lem: m<k and a_1=0}
Let $\Gamma$ be a $2$-walk-regular graph with valency $k$. If $\G$ has a $2$-walk-regular idempotent for an eigenvalue $\theta \neq \pm k$ with rank less than $k$, then the intersection number $a_1$ is positive.
\end{lemma}
\begin{proof}
Assume that $a_1=0$, and so $b_1=k-1$. Let $x$ be a vertex of $\Gamma$ and let $\Delta$ be the subgraph of $\Gamma$
induced on the neighbors of $x$. Note that the local graph $\Delta$ has no edges and hence $\Delta$ has only $0$ as an
eigenvalue. As (the rank) $m<k$, by Proposition~\ref{thm: lacal 2} we know that $-1-\frac{b_1}{\theta+1}$
is an eigenvalue of $\Delta$, which shows that $\theta=-k$. This is a contradiction, so $a_1$ should be positive.
\end{proof}

\begin{theorem}\label{thm: multi 3} Let $\Gamma$ be a $2$-walk-regular graph different from a
complete multipartite graph, with valency $k\geq3$ and eigenvalue $\theta\neq\pm k$ with multiplicity $3$. Then
$\Gamma$ is a cubic graph with $a_1=a_2=0$ or a distance-regular graph. Moreover, if $\Gamma$ is distance-regular, then
$\Gamma$ is the cube, the dodecahedron, or the icosahedron.
\end{theorem}
\begin{proof}
As $\Gamma$ is $2$-walk-regular and not complete multipartite, we have that $k\leq\frac{(m+2)(m-1)}{2}=5$ by
Theorem~\ref{thm: bound k with m}.

For $k=3$, if at least one of the intersection numbers $a_1$ and $a_2$ is not zero, then $b_2 \leq 1$ and this shows
that the graph $\Gamma$ is distance-regular by Proposition~\ref{prop: t-wr bt=1 DRG}. All cubic distance-regular graphs
are known (see \cite[Thm.~7.5.1]{bcn89}), and it follows that $\Gamma$ is either the cube or the dodecahedron.

For the cases with $k>3$, we first observe that $a_1>0$ by Lemma~\ref{lem: m<k and a_1=0}. Note also that $a_1<k-1$
because $\G$ is not complete.

Assume first that $k=5$. Let $x$ be a vertex of $\Gamma$ and let $\Delta$ be the subgraph of $\Gamma$ induced on the
neighbors of $x$. Then the number of edges in $\Delta$ equals $5a_1/2$, which implies that $a_1$ should be even, and
hence $a_1=2$. Now the local graph $\Delta$ is a pentagon, and hence $\Gamma$ is the icosahedron (see
\cite[Prop.~1.1.4]{bcn89}), which is indeed distance-regular with an eigenvalue with multiplicity $3$ (it has spectrum
$\{3^1,\sqrt{5}^3,1^5,0^4,-2^4,-\sqrt{5}^3\}$).

Next, we assume that $k=4$. In this case $a_1\in\{1,2\}$. If $a_1=2$, then the local graph $\Delta$ is a quadrangle and
hence the graph $\Gamma$ is the octahedron. The octahedron is distance-regular and has spectrum $\{4^1,0^3,-2^2\}$, but
it is a complete multipartite graph $K_{3\times2}$.

To complete the proof, we may assume that $k=4$ and $a_1=1$. As $a_1=1$, the local graph $\Delta$ is a disjoint union
of two edges, i.e., the spectrum of $\Delta$ is $\{1^2,-1^2\}$. By Proposition~\ref{thm: lacal 2} we know that
$-1-\frac{b_1}{\theta+1}$ is an eigenvalue of $\Delta$, so $\theta=-2$ because $b_1=2$. Because every vertex is in
precisely two triangles, and every edge is precisely one triangle, it will be useful to consider the triangle-vertex
$(0,1)$-incidence matrix $N$, where $N_{T,x}=1$ if vertex $x$ is in triangle $T$. If $B$ is the adjacency matrix of
$\G$, then $B=N^{\top}N-2I$ (in fact, $\G$ is a line graph, cf.~Proposition \ref{prop: line graph t-w-r}). If $n$ is
the number of vertices and $c$ the number of triangles, then $3c=2n$. Because the rank of $N$ is at most $c$, it
follows that $\G$ has eigenvalue $-2$ with multiplicity at least $n-c=n/3$, which implies that $n \leq 9$. Consider now
the intersection number $c_2$ and the number of vertices $k_2$ at distance $2$ from a fixed vertex in $\G$. Because
$c_2k_2=kb_1=8$ and $\G$ is not complete multipartite, it follows that $c_2 \in \{1,2\}$. If $c_2=1$, then $n \geq
1+k+k_2 =13$, a contradiction. If $c_2=2$, then $n \geq 1+k+k_2=9$, with equality if and only if $\G$ is strongly
regular with parameters $(9,4,1,2)$. This implies that $\G$ is the lattice graph $L_2(3)$, which however has no
eigenvalue with multiplicity $3$. So, there is no possible graph for $k=4$.
\end{proof}

Contrary to many of the other results, it seems difficult to generalize the condition on the eigenvalue multiplicity in
Theorem \ref{thm: multi 3} to a condition on the rank of a $2$-walk-regular idempotent. The multiplicity condition
gives a bound on $n$, and without this bound the case $k=4,a_1=1,b_2=2$ causes difficulties.

Notice that the complete multipartite graph $K_{(m+1)\times \omega}$ has eigenvalue $-\omega$ with multiplicity $m$,
and hence the complete multipartite graph $K_{4\times \omega}$ has an eigenvalue with multiplicity $3$. The only
complete multipartite graph having eigenvalue $0$ with multiplicity $3$ is the earlier mentioned $K_{3\times2}$.
Examples of other 2-walk-regular graphs (not being distance-regular) with an eigenvalue with multiplicity $3$ can be
found in the Foster census of symmetric cubic graphs \cite{FosterCensus}, such as the graphs $F056A$, $F060A$, $F104A$,
$F112C$, as well as the generalized Petersen graphs $G(8,3)$, $G(12,5)$, and $G(24,5)$, which correspond to graphs
$F016A$ (also known as the M\"{o}bius-Kantor graph; the double cover of the cube without quadrangles), $F024A$, and
$F048A$, respectively. In a forthcoming paper, we will present an infinite family of
$2$-walk-regular graphs with a multiplicity $3$.

\section{The Delsarte bound and geometric graphs}\label{sec: geometric}
In this section we start by observing that the Delsarte bound \cite{D73} for the size of a clique also holds for
$1$-walk-regular graphs. We will in fact prove a somewhat stronger statement and study the cases when equality is
attained. After that, we will focus our attention on the highly related notion of geometric graphs. We will show that
there are finitely many non-geometric $2$-walk-regular graphs with bounded smallest eigenvalue and fixed diameter.

\subsection{The Delsarte bound}\label{sec:delsartebound}

\begin{proposition}\label{prop: bound w with theta}
Let $\G$ be a connected $k$-regular graph with a $1$-walk-regular idempotent $E$ for an eigenvalue $\theta<0$. If $C$ is a clique in $\G$ with characteristic vector
$\chi$, then $|C| \leq 1 - \frac{k}{\theta}$, with equality if and only if $E\chi=0$.
\end{proposition}

\begin{proof}
Let $C$ be a clique of $\Gamma$ of size $c$. As in Proposition \ref{prop:cliquerank}, the principal submatrix of $E$
indexed by the vertices of $C$ equals $\alpha_0I+\alpha_1 (J-I)$. Since $E$ is positive semidefinite, it follows that
$0 \leq \chi^{\top}E\chi = c(\alpha_0+\alpha_1(c-1))$. Now the bound on $c$ follows by using \eqref{eq: relation alpha0
alpha1}: $\alpha_0/\alpha_1=k/\theta$. If equality holds, then $0 = \chi^{\top}E\chi =\chi^{\top}E^2\chi
=\|E\chi\|^2$, so $E\chi=0$ (and the other way around).
\end{proof}

We call a clique with size attaining this bound a {\em Delsarte clique}. Note that if the multiplicity of $\theta$
equals $|C|-1$, that is, the bound of Proposition \ref{prop:cliquerank} is tight, then $C$ is a Delsarte clique.
Clearly, Proposition \ref{prop: bound w with theta} applies to $1$-walk-regular graphs, so that we obtain the following
Delsarte bound.

\begin{theorem}\label{thm: Delsarte 1-walk reg}
Let $\Gamma$ be a $1$-walk-regular graph with valency $k$ and smallest eigenvalue $\theta_d$. Then every
clique of $\Gamma$ has at most $1-\frac{k}{\theta_d}$ vertices.
\end{theorem}

We remark that if the graph is $1$-walk-regular, then equality in Proposition~\ref{prop: bound w with theta} can only
occur for $\theta=\theta_d$. Line graphs of regular graphs with valency at least $3$ constitute a class of graphs for
which the bound is satisfied with equality. However, the minimal idempotent corresponding to its smallest eigenvalue
does not necessarily satisfy the conditions of Proposition \ref{prop: bound w with theta}. On the other hand, the
Cartesian product $K_m \oplus K_n \oplus K_p$ of three complete graphs (a generalized Hamming graph) is
$0$-walk-regular with maximal cliques of size $m, n$, and $p$, while the Delsarte `bound' equals $(m+n+p)/3$, so for
particular values of $m,n$, and $p$, it has maximal cliques of size attaining the Delsarte bound, but also larger
cliques. A final remark is that the same approach works for bounding the maximum number of vertices mutually at
distance $t$ in a $t$-walk-regular graph.

In a distance-regular graph with diameter $D$, a Delsarte clique $C$ has covering radius (that is, the maximum distance
of a vertex to the clique) equal to $D-1$ (note that in every connected graph with diameter $D$, the covering radius of
a clique is either $D-1$ or $D$). Moreover, $C$ is {\em completely regular} in the sense that every vertex at distance
$i$ from $C$ is at distance $i$ from the same number of vertices $\phi_i$ of $C$ (and hence it is at distance $j$ from
the same number of vertices $\phi_{i,j}$ of $C$ for every $j$), for $i=0,\dots,D-1$ . We can generalize this as
follows.

\begin{proposition}\label{prop:coveringradius}
Let $\Gamma$ be a $1$-walk-regular graph with $d+1$ distinct eigenvalues, and let $C$ be a
Delsarte clique. Then the covering radius of $C$ is at most $d-1$. Moreover, if $\G$ is $t$-walk-regular, then every
vertex at distance $i$ from $C$ is at distance $i$ from the same number of vertices $\phi_i$ of $C$, for
$i=0,\dots,t-1$.
\end{proposition}

\begin{proof} Let $E$ be the minimal idempotent for the smallest eigenvalue $\theta_d$. By Proposition \ref{prop: bound w with theta}, we have that $E\chi=0$. Consider the adjacency algebra
$\AL=\langle I,A,A^2,\dots,A^d \rangle$ of $\G$. Because $E$ is a nonzero matrix in $\AL$, it follows that $\AL \chi$
has dimension at most $d$. This implies that if $(A^d\chi)_x \neq 0$, then $(A^i\chi)_x \neq 0$ for some $i<d$, hence
every vertex is at distance at most $d-1$ from $C$.

Now let $c$ be the size of $C$ and assume that $\G$ is $t$-walk-regular. Let $x$ be a vertex at distance $i$ from $C$,
for $i<t$. Let $\phi_i$ be the number of vertices of $C$ that are at distance $i$ from $x$ (we intend to show that this
number does not depend on $x$), then the number of vertices of $C$ that are at distance $i+1$ from $x$ equals
$c-\phi_i$. Because $(E\chi)_x=0$, it follows that $\phi_i\alpha_i+(c-\phi_i)\alpha_{i+1}=0$. Indeed, this shows that
$\phi_i$ does not depend on $x$, as long as $\alpha_i \neq \alpha_{i+1}$. Suppose however that $\alpha_i =
\alpha_{i+1}$. Then it follows that $\alpha_{i+1}=0$, and by repeatedly using \eqref{eq: relations alpha's}, it follows
that $\alpha_j=0$ for all $j \leq i+1$, in particular $\alpha_0=0$, a contradiction.
\end{proof}

In the above we used the original definition of completely regular codes by Delsarte \cite{D73}. Neumaier's
\cite{Neumaier} alternative definition (which is equivalent for codes in distance-regular graphs) is in terms of the
distance partition with respect to the code being equitable, and also this property can be adjusted to Delsarte cliques
in $t$-walk-regular graphs in a straightforward manner; cf.~Neumaier \cite[Thm.~4.1]{Neumaier}. Note however that for
arbitrary codes (not just cliques), the two concepts of `partial complete regularity' are not equivalent. `Partially
Delsarte completely regular' seems to be stronger than `partially Neumaier completely regular' (at least in
distance-regular graphs), in the same way as $t$-walk-regularity is stronger than $t$-partial distance-regularity. An
example showing this is the code consisting of two vertices at distance $n-1$ in the $n$-cube.

\subsection{Geometric graphs}\label{secsub:geometric}

A graph is {\em geometric} if there exists a set of Delsarte cliques such that every edge lies on exactly one of them.
The notion of geometric graph in this sense was introduced by Godsil \cite{G93} for distance-regular graphs. Examples
of geometric graphs are bipartite graphs (trivially) and line graphs of a regular graphs with valency at least 3.

Koolen and Bang \cite{KB10} proved that there are only finitely many non-geometric distance-regular graphs with
smallest eigenvalue at least $-\omega$ and diameter at least $3$. It is also possible to state a similar result for
$2$-walk-regular graphs. More precisely, Koolen and Bang \cite[Thm.~3.3]{KB10} showed that there are finitely many
distance-regular graphs with smallest eigenvalue $-\omega$, diameter $D\geq 3$, and small $c_2$ (compared with $a_1$).
In order to prove this, they bound the valency $k$ using Godsil's multiplicity bound (the analogue of Theorem \ref{thm:
bound k with m}), using the multiplicity of the second largest eigenvalue $\theta_1$. In turn, a bound on $m(\theta_1)$
is derived from the analogue of Proposition~\ref{thm: lacal 2}, after showing that $m(\theta_1)<k$. One of the key
points for the latter inequality is to give an upper bound for the number of vertices in $\Gamma$. Their argument,
however, does not apply to $2$-walk-regular graphs. The following lemma intents to solve this problem.

\begin{lemma}\label{lem:boundv}
Let $\omega\geq 2$ be an integer. Let $\Gamma$ be a $2$-walk-regular graph with valency $k$, diameter $D$, and smallest
eigenvalue at least $-\omega$. If $\epsilon$ is such that $0<\epsilon<1$ and $c_2\geq a_1\epsilon$, then
$|V|<\left(\frac{2\omega^2}{\epsilon}\right)^D Dk.$
\end{lemma}
\proof Let $\Delta$ be the subgraph of $\Gamma$ induced on the neighbors of $x$. Observe first that the size of a
coclique in $\Delta$ is at most $\omega^2$, because otherwise by eigenvalue interlacing, $\Gamma$ would have an
eigenvalue smaller than $-\omega$. Because the number of vertices in $\Delta$ is $k$, it follows that $k\leq
\omega^2(a_1+1)$.

Note that the assumptions on $\epsilon$ imply that $\frac{a_1+1}{c_2}<\frac2{\epsilon}$. Since $b_1=k-1-a_1$, this
implies that
\begin{equation*}
\frac{b_1}{c_2} \leq \frac{(\omega^2-1)(a_1+1)}{c_2}< \frac{2\omega^2}{\epsilon}.
\end{equation*}
Fix $x\in V$ and let $k_i=|\Gamma_i(x)|$ ($i=0,\dots,D$). We claim that $k_{i+1}\leq \frac{b_1}{c_2}k_i$ for
$i=1,\dots,D-1$. In order to show this claim, first observe the following. For $y\in \G_{i+1}(x)$, let
$c_{i+1}(y)=|\Gamma(y)\cap \Gamma_{i}(x)|$. By taking $z\in \Gamma_{i-1}(x)$ with $\mbox{dist}(y,z)=2$, and observing
that $\Gamma(y)\cap \Gamma(z) \subset \Gamma(y)\cap \Gamma_{i}(x)$, it follows that $c_2 \leq c_{i+1}(y)$. Similarly,
one can show that $b_i(y) \leq b_1$ for $y \in \G_i(x)$, where $b_i(y)=|\Gamma(y)\cap \Gamma_{i+1}(x)|$. Now
$k_{i+1}c_2 \leq \sum_{y\in \Gamma_{i+1}(x)}c_{i+1}(y)=\sum_{y\in \Gamma_{i}(x)}b_{i}(y) \leq k_ib_1$, which proves the
claim. Thus,
\[\pushQED{\qed}
|V|=\sum_{i=0}^D k_i< 1+\sum_{i=1}^D\left(\frac{2\omega^2}{\epsilon}\right)^{i-1}k<\left(\frac{2\omega^2}{\epsilon}\right)^D Dk.
\qedhere
\popQED\]\\

As a consequence of this lemma, the proof by Koolen and Bang \cite{KB10} also applies to 2-walk-regular graphs, so we
have the following result (cf.~\cite[Thm.~3.3]{KB10}).

\begin{theorem}\label{thm:bangkoolen}
Let $0<\epsilon<1$, and let $\omega\geq 2$ and $D\geq 3$ be integers. Let $\Gamma$ be a $2$-walk-regular graph with
valency $k$, diameter $D$, smallest eigenvalue at least $-\omega$, and with $c_2\geq a_1\epsilon$. Then
$k<D^2\left(\frac{2\omega^2}{\epsilon}\right)^{2D+4}$. In particular, there are finitely many such graphs.
\end{theorem}

Next is to show, as it happens with distance-regular graphs (see Koolen and Bang \cite[Thm.~5.3]{KB10}), that if $a_1$
is large enough (compared to $c_2$), then a $2$-walk-regular graph with smallest eigenvalue at least $-\omega$ is
geometric. The next result by Metsch \cite{M99} is a key point for that purpose.

\begin{proposition}\cite[Result 2.1]{M99}\label{thm: Metsch}
Let $k\geq 2$, $\mu\geq 1$, $\lambda \geq 0$, and $s\geq 1$. Suppose that $\Gamma$ is a regular graph with valency $k$
such that every two non-adjacent vertices have at most $\mu$ common neighbors, and every two adjacent vertices have
exactly $\lambda$ common neighbors. Define a line as a maximal clique in $\Gamma$ with at least
$\lambda+2-(s-1)(\mu-1)$ vertices. If $\lambda> (2s-1)(\mu-1)-1$ and $k<(s+1)(\lambda+1)-s(s+1)(\mu-1)/2$, then every
vertex is in at most $s$ lines, and each edge lies in a unique line.
\end{proposition}

\begin{proposition}\label{prop:omegageometric}
Let $\omega\geq 2$ be an integer and let $\Gamma$ be a $2$-walk-regular graph with valency $k$, diameter $D\geq 2$, and
smallest eigenvalue in the interval $[-\omega, 1-\omega)$. If $a_1> \omega^4 c_2$, then $\Gamma$ is geometric.
\end{proposition}
\begin{proof}
Let $s=\omega^2$, $\lambda=a_1$, $\mu=c_2$, and define a line as in Proposition~\ref{thm: Metsch}. Using that $k\leq
\omega^2(a_1+1)$, we can check that the conditions of Proposition~\ref{thm: Metsch} are fulfilled. So, every vertex is
in at most $\omega^2$ lines and every edge is in a unique line, where the size of a line is at least
$$
a_1+2-(\omega^2-1)(c_2-1)>\omega^4c_2-(\omega^2-1)(c_2-1)=\omega^2(\omega^2c_2-c_2+1)+c_2-1\geq \omega^2+1.
$$
From Theorem~\ref{thm: Delsarte 1-walk reg}, we know that the size of a clique is less than $1+\frac{k}{\omega-1}$. Let
$M_x$ be the number of lines through the vertex $x$ and let $\omega_x+1$ be the average size of these lines. Then
$M_x=\frac{k}{\omega_x}> \frac{k}{k/(\omega-1)}=\omega-1$, hence
\begin{equation}\label{eq: bound number lines}
M_x \geq \omega.
\end{equation}
Let $\mathcal{C}$ denote the set of lines, and let $f$ be the number of flags (incident vertex-line pairs) $(x,C)$. By
counting these pairs in two ways, we obtain that $|V|\omega^2\geq f \geq |\mathcal{C}|(\omega^2+1)$, and hence there
are more vertices than lines.

Now let $N$ be the vertex-line incidence matrix of $\Gamma$. Then $NN^{\top}=A+D$, where $D$ is a diagonal matrix with
$D_{xx}=M_x$; moreover, $NN^{\top}$ is singular. Because $z^{\top}Az\geq -\omega$ and $z^{\top}Dz \geq \omega$ for
every vector $z$ of length $1$, it follows that if $z$ is an eigenvector of $NN^{\top}$ with eigenvalue $0$, then
equality holds in both inequalities. In particular, this implies that there is a vertex $x$ with $M_x=\omega$. It
follows that for this vertex $x$ we have $\omega_x=\frac{k}{\omega}$ (and all lines
through $x$ are Delsarte cliques), hence $a_1\geq \frac{k}{\omega}-1$, or equivalently, $k\leq \omega(a_1+1)$.

Now that we have a better bound for $k$, we can apply again Proposition~\ref{thm: Metsch}, but this time we can set
$s=\omega$. Now every vertex is in at most $\omega$ lines, so we must have equality in \eqref{eq: bound number lines}.
Thus, every vertex is in exactly $\omega$ lines, and hence every line is a Delsarte clique. So $\G$ is geometric.
\end{proof}

As a consequence of Theorem \ref{thm:bangkoolen} and Proposition \ref{prop:omegageometric}, we have the following
result.

\begin{theorem}\label{thm: finitely many non-geometric}
Let $\omega\geq 2$ and $D\geq 3$. There are finitely many non-geometric $2$-walk-regular graphs with diameter $D$ and
smallest eigenvalue at least $-\omega$.
\end{theorem}
Let us remark that we need to fix both $\omega$ and $D$ for the finiteness. Conder and Nedela
\cite[Prop.~2.5]{ConderNedela} constructed infinitely many $3$-arc-transitive cubic graphs with girth $11$. Because a
geometric graph without triangles must be bipartite, this shows that there are infinitely many non-geometric
$3$-walk-regular graphs with smallest eigenvalue larger than $-3$. To show that we need to fix $\omega$, we consider
the symmetric bilinear forms graph. This graph has as vertices the symmetric $n \times n$ matrices over $\F_q$, where
two vertices are adjacent if their difference has rank $1$; see \cite[Sec.~9.5.D]{bcn89}. For $q$ even and $n \geq 4$,
this graph is not distance-regular, but it is $2$-walk-regular. For $n=4$, these graphs have diameter $5$, and one can
show using the distance-distribution diagram (see \cite[p.~22]{BCNcoradd}) that the smallest eigenvalue equals
$-1-q^3$. Because the valency equals $q^4-1$, this graph cannot be geometric, even though there are `lines' of size
$q$, but these are not Delsarte cliques. Note that also for $n>4$ (and $q$ even), we can show that the symmetric bilinear forms graph is not geometric, but we omit the proof, because it is rather involved.

On the other hand, we need $2$-walk-regularity, because the earlier mentioned $2$-coclique extensions of the lattice
graphs provide an infinite family of non-geometric $1$-walk-regular graphs with diameter $2$ and smallest eigenvalue
$-4$. Theorem~\ref{thm: finitely many non-geometric} thus illustrates once more the important structural gap between
$1$- and $2$-walk-regular graphs.

Note finally that a geometric graph $\G$ is the point graph of the partial linear space of vertices and (some) Delsarte
cliques, and that one can consider also the dual graph on the cliques, that is, the point graph of the dual of this
partial linear space. In particular when $\G$ is locally a disjoint union of cliques (i.e., when
$k=-\theta_d(a_1+1)$), this can be used to obtain new examples of $t$-walk-regular graphs, in the same spirit as in
Proposition \ref{prop: line graph t-w-r}, although now one has to consider the so-called geometric girth instead of the
usual girth. For example, the Hamming graphs have geometric girth $4$ (as $c_2>1$), and the dual graphs of the Hamming
graph (with diameter at least three) are only $1$-walk-regular. The distance-regular near octagon coming from the
Hall-Janko group (see \cite[Sec.~13.6]{bcn89}) has geometric girth $6$ and its dual is $2$-walk-regular.

We finish by observing that besides distance-regular graphs and the above mentioned symmetric bilinear forms graphs, we do not know of many examples of $2$-walk-regular graphs with $c_2 \geq 2$. We challenge the reader to construct more
such examples.

\section*{Acknowledgments}

The authors thank the referee for his careful reading and helpful comments and Akihiro Munemasa for pointing the authors to the fundamental bound.

\bibliographystyle{abbrv}

\bibliography{bibfile}
\end{document}